\documentclass[11pt]{amsart}

\usepackage{amssymb,amsmath,amscd,graphicx,
latexsym,amsthm}
\usepackage{amssymb,latexsym,amsmath,amscd,graphicx}
\usepackage[mathscr]{eucal}
  \usepackage[all]{xy}
\def\PP{{\mathbb P}}

\def\OO{{\mathcal O}}

\def\F{\mathcal{F}}
\def\E{\mathcal{E}}
\def\G{\mathcal{G}}

\def\N{\mathcal{N}}

\def\cP{\mathcal{P}}
\def\Pic0{{\rm Pic}^0}

\def\DD{{\mathbf{D}}}

\def\N{\mathcal N}

\theoremstyle{plain}

\newtheorem{theorem}{Theorem}[section]

\newtheorem{proposition/example}[theorem]{Proposition/Example}
\newtheorem{proposition}[theorem]{Proposition}

\newtheorem{corollary}[theorem]{Corollary}

\newtheorem{claim}[theorem]{Claim}

\theoremstyle{definition}
\newtheorem{definition}[theorem]{Definition}

\newtheorem{remark}[theorem]{Remark}
\newtheorem{example}[theorem]{Example}

\newtheorem{conjecture/question}[theorem]{Conjecture/Question}

\newtheorem{remark/definition}[theorem]{Remark/Definition}
\newtheorem{notation/assumptions}[theorem]{Assumptions/Notation}

\numberwithin{equation}{section}

 \theoremstyle{remark}

\begin{document}

\title{Fully faithful Fourier-Mukai functors and generic vanishing }

\author[G. Pareschi]{Giuseppe Pareschi}
\address{Dipartimento di Matematica, Universit\`a di Roma, Tor Vergata, V.le della
Ricerca Scientifica, I-00133 Roma, Italy} \email{{\tt
pareschi@mat.uniroma2.it}}
\dedicatory{In memory of Alexandru T. Lascu}

\begin{abstract}  The aim of this  mainly expository note is to point out that, given an Fourier-Mukai functor, the condition making it fully faithful is  an instance of \emph{generic vanishing}. We test this point of view on some fairly classical examples, including the strong simplicity criterion of Bondal and Orlov, the standard flip and  the Mukai flop.
\end{abstract}

\thanks{}

\maketitle


\setlength{\parskip}{.1 in} 

The aim of this  mainly expository note is to point out that, given an Fourier-Mukai functor, the condition making it fully faithful is  an instance of \emph{generic vanishing}. We test this point of view on some fairly classical examples, including the strong simplicity criterion of Bondal and Orlov, the standard flip and  the Mukai flop.

The notion of generic vanishing arose in work of Green and Lazarsfeld on irregular varieties (\cite{gl1},\cite{gl2}), where they showed that the sheaves of holomorphic differential forms,  twisted with a generic topologically trivial line bundle,   satisfy a cohomological vanishing of Kodaira-Nakano type. 
The natural environment for the notion of generic vanishing introduced by Green and Lazarsfeld  is the Fourier-Mukai functor defined by the Poincar\'e line bundle. It makes sense to study the same kind of property for any FM functor (\cite{pp4},\cite{popa}). 

In this note we remark that a FM functor $\Phi_\E^{X\rightarrow Y}: D^b(X)\rightarrow D^b(Y)$ is fully faithful if and only if $\OO_Y$, the structure sheaf of $Y$,  satisfies Green-Lazarsfeld's generic vanishing condition (in the current terminology: is a geometric GV-object) with respect to the FM functor 
\[\Phi^{Y\rightarrow X\times X}_{\E\boxtimes_Y\E^\vee}:D^b(Y)\rightarrow D^b(X\times X)\]
 (see below for the notation), plus an additional condition which is usually easier to check. In essence, to be fully faithful is very close to be a generic vanishing condition. While this is certainly not a new result, but just a restatement of  well known  basic facts, it is the author's hope that this point of view can be an useful complement to the existing methods of investigating whether a given FM functor is fully faithful, in particular an equivalence. We test this by  providing different proofs of some fairly classical full-faithfulness results.

Here is what the reader will find in this paper. The first section is background on fully faithful FM functors. The second section is background about generic vanishing conditions: they are usually stated in three equivalent ways, which we recall. In Section 3 we show that, in the context of full-faithfulness, the first equivalent condition essentially boils down to the strong simplicity criterion of Bondal and Orlov. Interestingly, the natural version of Bondal-Orlov's criterion in this context works under weaker hypotheses (Prop. \ref{mio}). In Section 4 we use the second equivalent condition to give, or outline, alternative proofs about full-faithfulness of the natural FM functors associated to standard flips and Mukai flops. Finally, in the last section we prove a full-faithfulness criterion (Prop. \ref{c}) corresponding the third equivalent way of expressing generic vanishing, and we illustrate it in the example of Poincar\'e kernels. 

 \noindent\textbf{Notation.} (a) All functors (as $f^*$, $f_*$,  $\otimes$, ....) denote the  functor on the derived category of coherent sheaves. For example $\otimes$ means $\otimes^\mathbf{L}$, and the underived tensor product of coherent sheaves is denoted $tor_0$. Moreover $H^i$ means hypercohomology.\\
 (b) Unless otherwise stated, all varieties are assumed smooth and projective over an algebraically closed  ground field.
 
\vskip0.4truecm \noindent This paper is dedicated to the memory of Alexandru Lascu, my teacher and adviser back when I was an undergraduate at the University of Ferrara.

\section{Fully faithful Fourier-Mukai functors }

 Let $X$ and $Y$ be smooth projective varieties  and
\[\Phi^{X\rightarrow Y} _\E :D^b(X)\rightarrow D^b(Y)\]
a Fourier-Mukai functor of kernel $\E\in D^b(X\times Y)$. We denote 
\begin{equation}\label{dual}\E^\vee=R\mathcal{H}om(\E,\OO_{X\times Y})
\end{equation}
 and $p_X$ and $p_Y$ the projections of $X\times Y$. 
The functor
\[ \Phi^{Y\rightarrow X} _{\E^\vee\otimes p_Y^*\omega_Y[\dim Y]}: D^b(Y)\rightarrow D^b(X)
\]
is the left adjoint of $\Phi_\E$. It follows that $\Phi_\E$ is fully faithful if and only the natural morphism of functors
\begin{equation}\label{left} \Phi^{Y\rightarrow X} _{\E^\vee\otimes p_Y^*\omega_Y[\dim Y]}\circ \Phi^{X\rightarrow Y}_\E\longrightarrow id_{D^b(X)}
\end{equation} 
is an isomorphism.
The functor
$\Phi^{Y\rightarrow X} _{\E^\vee\otimes p_Y^*\omega_Y[\dim Y]}\circ \Phi^{X\longrightarrow Y}_\E$ is the FM functor  of kernel
\[\Phi^{Y\times Y\rightarrow X\times X}_{\E\boxtimes (\E^\vee\otimes p_Y^*\omega_Y)[\dim Y]}(\OO_{\Delta_Y})
\]
(e.g. \cite{huybrechts} Ex. 5.13(ii)). Therefore, since the unique kernel for $id_{D^b(X)}$ is $\OO_{\Delta_X}$,  the functor $\Phi_\E^{X\rightarrow Y}$ is fully faithful if and only if
\begin{equation}\label{left'}\Phi^{Y\times Y\rightarrow X\times X}_{\E\boxtimes (\E^\vee\otimes p_Y^*\omega_Y)}(\OO_{\Delta_Y})=\OO_{\Delta_X}[-\dim Y]
\end{equation}

Given $\E$ and $\F$ objects of $D^b(X\times Y)$, let $\E\boxtimes_Y\F$ be the object of $D^b(X\times X\times Y)$ defined as 
 \[\E\boxtimes_Y\F={p_{13}}^*\E\otimes{p_{23}}^*\F\]
  where $p_{13}$ and $p_{23}$ are the two projections of $(X\times Y)\times_Y (X\times Y)=X\times X\times Y$. 
 Since $\Phi_{\E\boxtimes \F}^{Y\times Y\rightarrow X\times X}(\OO_{\Delta_Y})=\Phi^{Y\rightarrow X\times X}_{\E\boxtimes_Y \F}(\OO_Y)$ the above condition 
 (\ref{left'})  can be also written as follows
  \begin{equation}
 \Phi^{Y\rightarrow X\times X}_{\E\boxtimes_Y (\E^\vee\otimes p_Y^*\omega_Y)}(\OO_Y)=\OO_{\Delta_X}[-\dim Y]
 \end{equation}
 or also
 \begin{equation}\label{left2}
 \Phi^{Y\rightarrow X\times X}_{\E\boxtimes_Y \E^\vee}(\omega_Y)=\OO_{\Delta_X}[-\dim Y]
 \end{equation}

 \section{Generic vanishing }

 Let $Y$ and $Z$ be  smooth projective varieties and $\cP\in D^b(Y\times Z)$.  
 Let  $z\in Z$ be a closed point and $k_z$ its residue field, seen a coherent sheaf on $Z$ supported at $z$. We have that $\Phi_{\cP}^{Z\rightarrow Y}(k_z)=i_z^*\cP$, where $i_z: Y\rightarrow Y\times Z$ is $i_z(y)=(y,z)$.
 We consider now the Fourier-Mukai functor in the opposite direction, $\Phi_{\cP}^{Y\rightarrow Z}:D^b(Y)\rightarrow D^b(Z)$. Given $F\in D^b(Y)$, we define its cohomological support loci with respect to $\Phi_{\cP}^{Y\rightarrow Z}$ as 
 \[V^i_\cP(Y,F)=\{z\in Z\>|\> h^i(Y, F\otimes \Phi^{Z\rightarrow Y}_{\cP}(k_z))>0\}\]
 
 \begin{example}\label{irregular}(\emph{Green-Lazarsfeld sets}) Let $Y$ be an irregular variety, $Z=\Pic0 Y$ and $\cP$ a Poincar\'e line bundle on $Y\times \Pic0 Y$. Then $z\in \Pic0Y$ corresponds to a line bundle $P_z$ on $Y$, which is precisely $\Phi_{\cP}^{\Pic0Y\rightarrow Y}(k_z)$.  Given a coherent sheaf $F$ on $Y$, the cohomological support loci 
 \[V^i_{\cP}(Y,F)=\{z\in\Pic0 Y\>|\> h^i(F\otimes P_z)>0\>\}
 \]
 were introduced and studied by Green and Lazarsfeld for the sheaves $ F=\Omega^j_Y$ (\cite{gl1}, \cite{gl2}) and subsequently studied for other relevant sheaves on abelian and/or irregular varieties, see e.g. \cite{pp1}, \cite{pp2}.
 \end{example}
 The following notion  was  introduced by Mihnea Popa in \cite{popa}, Def. 3.7. 
 
 \begin{definition} (\emph{Geometric GV-objects})
An object $F\in D^b(Y)$ is called \emph{a geometric GV-object} with respect to a functor $\Phi_\cP^{Y\rightarrow Z}$ if\\
 (i) $V^i_\cP(Y, F)=\emptyset$ for $i<0$ and\\
 (ii) $\mathrm{codim}_ZV^i_\cP(Y, F)\ge i$ for $i\ge 0$ \footnote{Note that in \emph{loc cit} condition (i) is stated in a different way, namely  $R^j\Phi^{X\times Y}_{\E^\vee\otimes p_Z^*\omega_Z}(\F^\vee\otimes \omega_Y)=0$ for $j>\dim Y$. This is equivalent to (i) by duality and base-change. Indeed by duality  (i) is equivalent to the fact the loci $V^j_{\cP^\vee\otimes p_Z^*\omega_Z}(Y, \F^\vee\otimes\omega_Y)$ are empty for all $j>\dim Y$, which is in turn equivalent (by an easy application of base-change) to the vanishing of the sheaves $R^j\Phi^{X\times Y}_{\E^\vee\otimes p_Z^*\omega_Z}(\F^\vee\otimes \omega_Y)=0$ for all $j>\dim Y$}.
  \end{definition}

The following result is well known, see \cite{pp3}, \cite{pp4} and especially \cite{popa}, Th, 3.8, Remark 3.10 and Cor. 4.3 and references therein. See also Remark \ref{attributions} below.

 \begin{theorem}\label{gv} In the above setting, the following are equivalent
 
 \noindent (a) $F$ is a geometric GV-object with respect to the functor $\Phi_\cP^{Y\rightarrow Z}$;
 
 \noindent (b) $\Phi^{Y\rightarrow Z}_{\cP^\vee}(F^\vee\otimes \omega_Y)$ is concentrated in cohomological degree $\dim Y$. That is:
 \[\Phi^{Y\rightarrow Z}_{\cP^\vee}(F^\vee\otimes \omega_Y)=R^{\dim Y}\Phi^{Y\rightarrow Z}_{\cP^\vee}(F^\vee\otimes \omega_Y)[-\dim Y]\]

  \noindent (c) If  $A$ is a sufficiently high multiple of an ample line bundle on $Z$ then 
  \[H^i(Y, \, \Phi^{Z\rightarrow Y}_{\cP^\vee}(A)\otimes  F^\vee\otimes \omega_Y)=0\quad\hbox{ for all $i\ne \dim Y$.} \footnote{condition (c) is more usually expressed in the dual way, namely $H^i(Y,F\otimes \Phi_{\cP[\dim Z]}^{Z\rightarrow Y}(A^\vee))=0$ for $i\ne 0$ (see e.e. \cite{pp4} Cor.3.11(b)). By duality and Serre vanishing it is easily seen that the two formulations are equivalent }\]
\end{theorem}
 
 According to a terminology/notation due to Mukai, condition ($b$) is sometimes referred to as the fact that $F^\vee\otimes\omega_Y$ \emph{satisfies the weak index theorem with index $i=\dim Y$}. For short: WIT($\dim Y$). If this is the  case the sheaf $R^{\dim Y}\Phi^{Y\rightarrow Z}_{\cP^\vee}(F^\vee\otimes \omega_Y)$ is denoted
$ \label{hat}\widehat{\F^\vee\otimes\omega_Y}$. In this notation condition (b) is written as
 \begin{equation}\label{hat}\Phi^{Y\rightarrow Z}_{\cP^\vee}(F^\vee\otimes \omega_Y)=\widehat{\F^\vee\otimes\omega_Y}[-\dim Y]
 \end{equation}

 \begin{remark}\label{smoothness}(\emph{On the assumptions for Theorem \ref{gv}})
 Theorem \ref{gv} works under more general hypotheses: assuming that the kernel $\cP$ is a perfect complex, for the equivalence between (a) and (b)  $Z$ need not to be projective, and both $Y$ and $Z$ need not to be smooth varieties, but only Cohen-Macaulay schemes of finite type over any field (but, if $Z$ is not Gorenstein, in condition (b) $\cP^\vee $ has to be replaced with $\cP^\vee\otimes p_Z^*\omega_Z$ see \cite{popa} Remark 3.10).  The equivalence with (c) holds under the further assumption that $Z$ is projective. We refer to \cite{popa} and \cite{pp4}.
\end{remark}

 \begin{remark}\label{duality}(\emph{Conditions $(a)$ and $(b)$ for (hyper)cohomology}) A simple-minded way to see the equivalence between ($a$) and ($b$) is as follows. Let $Z$ be a point, denoted $\{pt\}$, and $\cP=\OO_{Y\times\{pt\}}$. The two conditions of the previous theorem are reduced to:\\
 ($a_0$)  $H^i(Y,F)=0$ for $i\ne 0$;\\
 ($b_0$) $H^i(Y, F^\vee \otimes \omega_Y)=0$ for $i\ne \dim Y$. 
 
  \noindent They are equivalent by Serre duality, and the meaning of the equivalence between ($a$) and ($b$) is that, for arbitrary Fourier-Mukai functors, they admit  distinct equivalent generalizations: the generalization of ($a_0$) is geometric-GV, namely  the generic vanishing of a family of hypercohomology groups. The generalization of ($b_0$) is  WIT($\dim Y$), that is the vanishing of the  hyperdirect image sheaves $R^i\Phi_{\cP^\vee}(F^\vee\otimes p_Y^*\omega_Y)$. 
 \end{remark}
 
  \begin{remark}\label{attributions}(\emph{Perverse sheaves}) The geometric-GV and WIT($\dim Y$) conditions are better stated in terms of t-structures and perverse sheaves. We refer to \cite{popa} and \cite{popa-schnell} \S6-7 for this. Briefly, it follows from a result of Kashiwara (\cite{kas}) that a geometric GV-object with respect to $\Phi_\cP^{Y\rightarrow Z}$ is an object $F$  of $D^b(Y)$ such that $\Phi_\cP F$  belongs to the heart of the dual t-structure on $D^b(Z)$. This is the equivalence between (a) and (b) in Theorem \ref{gv}.   
 \end{remark}

 \begin{remark}\label{ample sequence}[\emph{Condition (c) with ample sequences}) By duality,
 \[(\Phi_{\cP^\vee}^{Z\rightarrow Y}(A))^\vee \cong \Phi_{\cP[\dim Z]}^{Z\rightarrow Y}(A^\vee\otimes \omega_Z)\]
 Threfore
  condition (c) can be written as follows
  \[\mathrm{Hom}(\Phi^{Z\rightarrow Y}_{\cP[\dim Z]}(A^{-1}\otimes \omega_Z), F^\vee\otimes\omega_Y[j])=0\qquad\hbox{ for  all $j\ne \dim Y$}\]
  Note that, if $A$ is ample and $k>>$ then $L_k:=A^{-k}\otimes \omega_Z$ is ample sequence in $coh(Z)$.
 Condition (c) of Theorem \ref{gv}  can be stated more generally as follows: given an ample sequence $\{L_k\}$ in $coh(Z)$, 
 \[\mathrm{Hom}(\Phi^{Z\rightarrow Y}_{\cP[\dim Z]}(L_k), F^\vee\otimes \omega_Y[j])=0\qquad\hbox{ for  all $j\ne \dim Y$ and $k<<$}\]  
The equivalence with condition (b) of Theorem \ref{gv} is proved in the same way.
 \end{remark}

 \section{Full faithfulness via condition $(a)$}
 
  The relationship between full faithfulness and generic vanishing is in (\ref{left2}), which can be reformulated as follows: \\
   \emph{$\Phi_\E$ is fully faithful if and only if $\omega_Y$ satisfies WIT($\dim Y$) with respect to $\Phi_{\E\boxtimes_Y\E^{\vee}}$ -- \emph{that is condition (b) of Theorem \ref{gv}} -- and, in addition,  its transform is the sheaf $\OO_{\Delta_X}$ in cohomological degree $\dim Y$. }
  
  Experience shows that, usually, the  more difficult part  to be checked is  the WIT($\dim Y$) condition, while the additional requirement is easier. With this in mind, Theorem \ref{gv} provides three distinct ways of checking full-faithfulness.
 
  Condition (a) leads to the classical strong simplicity criterion of  Bondal-Orlov (see \cite{bo}, \cite{br} and \cite{huybrechts} \S7.1). See also \cite{hls} and \cite{lopez} for generalizations). Actually one gets the result under  weaker hypotheses. To this purpose, let us  consider the loci
 \[W^i_\E(Y)=\{(x,x^\prime)\in X\times X\>|\> \mathrm{Hom}
 (\Phi^{X\rightarrow Y}_\E(k_x), \Phi^{X\rightarrow Y}_\E(k_{x^\prime}))[i])\ne 0\}\]
 
 \begin{proposition}\label{mio} Assume that $\mathrm{char} \, k=0$. Then $\Phi_{\E}^{X\rightarrow Y}:D^b(X)\rightarrow  D^b(Y)$ is fully faithful if and only if the following conditions hold:
 
 \noindent (a) $W^i_\E(Y)$ is empty for $i<0$;
 
\noindent (b)  $\dim W^i_\E(Y)\le 2\dim X-i$ for all $i\ge 0$;
 
\noindent (c)  $\mathrm{Hom} (\Phi^{X\rightarrow Y}_\E(k_x), \Phi^{X\rightarrow Y}_\E(k_{x^\prime}))= k$ if $x=x^\prime$ and
 $0$ otherwise.
 \end{proposition} 
 \begin{proof} Since $X$ is smooth
  \begin{equation}\label{smoothpoint}
  (\Phi^{X\rightarrow Y}_\E(k_x))^\vee\cong \Phi^{X\rightarrow Y}_{\E^\vee}(k_x)
  \end{equation} 
  for all $x\in X$. Therefore 
 \[\mathrm{Hom}
 (\Phi^{X\rightarrow Y}_\E(k_x), \Phi^{X\rightarrow Y}_\E(k_{x^\prime}))[i])
=\mathrm{Ext}^i(\Phi^{X\rightarrow Y}_\E(k_x), 
 \Phi^{X\rightarrow Y}_\E(k_{x^\prime})\cong \]
 \[\cong H^i(Y,\Phi^{X\rightarrow Y}_{\E^\vee}(k_x)\otimes \Phi^{X\rightarrow Y}_\E(k_{x^\prime}))=H^i(Y, \Phi^{X\times X\rightarrow Y}_
 {\E\boxtimes_Y
 \E^\vee}(k_{(x^\prime,x)}))\]
 It follows that
 \[W^i_\E(Y)=V^i_{\E\boxtimes_Y\E^\vee}(Y,\OO_Y)\]
 Therefore (a) and (b) mean exactly that $\OO_Y$ is a geometric GV-object with respect to $\Phi^{Y\rightarrow X\times X}_{\E\boxtimes_Y\E^\vee}$. By (a) $ \Leftrightarrow $ (b) of Theorem \ref{gv}, this is equivalent to the fact that $\Phi_{\E^\vee\boxtimes_Y \E}^{Y\rightarrow X\times X}(\omega_Y)$ is a coherent sheaf  on $X\times X$ in cohomological degree $\dim Y$ (note that $(\E\boxtimes_Y\E^\vee)^\vee\cong \E^\vee\boxtimes_Y\E$). According to notation (\ref{hat}), we denote $\widehat{\omega_Y}$ this coherent sheaf. Hypotheses (a) and (c) of the present Theorem imply, by cohomology and base change, that this sheaf is in fact a line bundle on a possibly non-reduced variety supported on the diagonal $\Delta_X$\footnote{in fact one known that, since $\OO_Y$ is geometric-GV with respect to $\Phi_{\cP}^{Y\rightarrow X\times X}$, the sheaf $\widehat \omega_Y$ has the "base-change property", namely, in the present case, the natural map $tor_0(\widehat\omega_Y,k_{(x,x\prime)})\rightarrow H^{\dim Y}(Y,\omega_Y\otimes \Phi_{\E^\vee\boxtimes_Y\E}^{X\times X\rightarrow Y}(k_{(x,x^\prime)}))$ is an isomorphism}. But in fact, as the ground field is assumed to be algebraically closed of characteristic zero, actually 
 \[\widehat{\omega_Y}={\delta_X}_*L\>,\] where $\delta_X:X\rightarrow X\times X$ is the diagonal embedding and $L$ is a line bundle on $X$: this is proved exactly as in Bridgeland's account of Bondal-Orlov's theorem, using the Kodaira-Spencer map (\cite{br} Lemmas 5.2-3 or \cite{huybrechts} Steps 3 and 5 or the proof of the main result in \cite{lopez}), so we won't reproduce this argument here. This already proves that $\Phi_\E$ is fully faithful (and, a posteriori, $L=\OO_X$). \end{proof}

 \begin{corollary}\label{bo} (Bondal-Orlov) $\Phi_\E$ is fully faithful if and only if 
\[\mathrm{Hom} (\Phi^{X\rightarrow Y}_\E(k_x), \Phi^{X\rightarrow Y}_\E(k_{x^\prime})[i])=\begin{cases} k&\hbox{if $x=x^\prime$ and $i=0$}\\
 0&\hbox{if $x\ne x^\prime$, or  $i<0$, or $i>\dim X$}\end{cases}\]
 \end{corollary}
 \begin{proof} The hypotheses can be restated as follows: $\mathrm{Hom} (\Phi^{X\rightarrow Y}_\E(k_x), \Phi^{X\rightarrow Y}_\E(k_{x^\prime}))= k$\  if $x=x^\prime$ and
 \[W^i_\E(Y) \begin{cases}=\emptyset& \hbox{for $i<0$}\\ \subseteq \Delta_X& \hbox{for $0\le i\le \dim X$}\\ =\emptyset& \hbox{for $i>\dim X$.}
 \end{cases}\]
  Therefore Proposition \ref{mio} implies the Corollary. \end{proof}
  
  \begin{remark}(\emph{On the assumptions for Prop. \ref{mio}  and Corollary \ref{bo}})  (i) As pointed out in \cite{hls} Remark 1.25 and \cite{lopez} the characteristic zero is necessary, unless one puts a supplementary hypothesis.\\
   (ii) Checking carefully  more general assumptions for te validity of  (\ref{left}) and for the equivalence between (a) and (b) in Theorem \ref{gv}, it follows that Prop. \ref{mio} and Corollary \ref{bo} work under more general hypotheses on $X$ and $Y$: $X$ needs to be smooth but not necessarily projective, while $Y$ needs to be projective but it is allowed to be singular (Cohen-Macaulay). This is a result in \cite{hls}, see also \cite{lopez} and references therein. 
  \end{remark}

  \section{Full faithfulness via condition (b)} 
  
  In this section we will consider some examples where, from the point of view of generic vanishing, the easiest way of proving/disproving full faithfulness is given by condition (\ref{left2}) at once, which corresponds to condition (b) of Theorem \ref{gv}. This amounts to 
  \begin{equation}\label{b}R^i\Phi_{\E\boxtimes_Y\E^\vee}(\omega_Y)=\begin{cases}0&\hbox{for $i\ne\dim Y$}\\
  \OO_{\Delta_X}&\hbox{for $i=\dim Y$}\\
  \end{cases}
  \end{equation}
  This is certainly a bit old fashioned (for example, it is the way Mukai originally showed in \cite{mukai} that the Poincar\'e kernel provides a derived equivalence between dual abelian varieties) but however the proofs below are easy, self-contained and conceptually clear, and might provide a complementary insight on some aspects of   Kawamata's conjecture K-equivalence $\Rightarrow$ D-equivalence.
  
  \begin{example} (\emph{Standard flip}) We consider a standard flip

  \begin{equation}\xymatrix{&&E=\mathbb{P}^l\times\mathbb{P}^k\ar[ddll]_{\pi_X}\ar@{^{(}->}[d]\ar[rrdd]^{\pi_Y}\\
  &&Z\ar[ld]_{p}\ar[rd]^{q}\\
  \mathbb{P}^l\>\>\>\ar@{^{(}->}[r]&X&&Y&\>\>\>\mathbb{P}^k\ar@{_{(}->}[l]\\}
  \end{equation}
  where $\N_{\PP^l/X}=\OO(-1)^{k+1}$ and $\N_{\PP^k/Y}=\OO(-1)^{l+1}$, so that the dimension of the varieties $X$, $Y$ and $Z$ is \  $d=k+l+1$.  The morphism $\pi_X$ (resp. $\pi_Y$) is the blow up of $\PP^l$ (resp. $\PP^k$).  Note that the functor $\Phi_{\OO_Z}^{X\rightarrow Y}$ coincides with ${q}_*\circ p^*$. The result, again due to Bondal and and Orlov \cite{bo},  is that: \ \ 
 \emph{ $\Phi_{\OO_Z}^{X\rightarrow Y}:D^b(X)\rightarrow D^b(Y)$ is fully faithful if and only if  $k\le l$. } 
 
  \noindent Let us prove this statement by verifying condition (\ref{b}).
We have that  
 \[\OO_{Z}^\vee=\omega_{Z}\otimes \omega_{X\times Y}^{-1}[-d]\]  
  It follows that in this case conditon (\ref{b}) takes the form
 \begin{equation}\label{b'}
 R^i\Phi_{\OO_Z\boxtimes_Y\omega_Z}(\OO_Y))=\begin{cases}0&\hbox{for $i\ne 0$}\\
  \delta_*(\omega_X)&\hbox{for $i=0$}\\
  \end{cases}
  \end{equation}
   Enumerating the four factors of the product $X\times X\times Y\times Y$, we denote $Z_{ij}$ the subvariety $Z$ of $X_i\times Y_j$. By definition $\OO_{Z}\boxtimes_Y\OO_Z$ is the following derived tensor product in $D(X\times X\times Y\times Y)$:
   \[\OO_{Z}\boxtimes_Y\OO_Z=(\OO_{Z_{13}}\boxtimes \OO_{Z_{24}})\otimes (\OO_{X\times X\times \Delta_{34}Y})\]
    The intersection in $X\times X\times Y\times Y$ of the two smooth and irreducible subvarieties $Z_{13}\times Z_{24}$ and $X\times X\times \Delta_{34}Y$ is
  the fibred product $Z\times_Y Z$. It has two irreducible components: 
   \begin{equation}\label{intersection}(Z_{13}\times Z_{24})\cap ( X\times X\times \Delta_{34}Y)=(\Delta_{13,24}\, Z)\cup (E\times_{\PP^k}E)=
   (\Delta_{12,34}\,Z)\cup (\PP^l\times\PP^l\times\Delta_{34}\PP^k)
    \end{equation}
   where $\Delta_{12,34} :(X\times Y)\hookrightarrow X\times X\times Y\times Y$ is the diagonal embedding $(x,y)\mapsto (x,x,y,y)$.  Sometimes we will simply write the right hand side as
    \[Z\,\cup\, \PP^l\times\PP^l\times\PP^k\]
    The two components are smooth and their intersection is $E=\Delta_{12,34}E=\Delta_{12}\PP^l\times \Delta_{34}\PP^k$. 
    
    The first component has the right codimension, namely $3d=2d+d$, whicle the second component has codimension $3d-(l-1)$ (hence it has the right codimension only for $l=1$). Since (\ref{intersection})  is the intersection of two smooth, hence locally complete intersection subvarieties of a smooth ambient variety, it follows that the higher $tor_i^{X\times X\times Y\times Y}(\OO_{Z_{13}\times Z_{24}},\OO_{X\times X\times\Delta_{34}\, Y})$ are non-zero only if $l>1$, and they are supported on $\PP^l \times\PP^l\times\PP^k$. Moreover $tor_1$ is locally free of rank $l-1$ and $tor_i=\wedge^i\,tor_1$ for $i\ge 1$. More precisely, we have the following
    \begin{claim}\label{claim} for $i>0$ \ $tor_i^{X\times X\times Y\times Y}(\OO_{Z_{13}\times Z_{24}},\OO_{X\times X\times\Delta_{34}\, Y})=\bigwedge^i\bigl(\OO_{\PP^l\times\PP^l\times\PP^k}(0,0,1)^{\oplus l-1}\bigr)$
    \end{claim}
    \begin{proof} We first compute 
    \[tor_i^{X\times X\times Y\times Y}(\OO_{E_{13}\times E_{24}},\OO_{X\times X\times\Delta_{34}\, Y})=tor_i^{X\times X\times Y\times Y}(\OO_{\PP^l\times\PP^l\times\PP^k\times \PP^k},\OO_{X\times X\times\Delta_{34}\, Y})=\]
    \[=p_{34}^*\N^\vee_{\PP^k/Y}=p_{34}^*\wedge^i(\OO(1)^{\oplus l+1})=\wedge^i(\OO_{\PP^l\times\PP^l\times\PP^k}(0,0,1)^{\oplus l+1})\]
    The third equality follows from the general isomorphism  $(\F\boxtimes\G)\otimes_{Y\times Y} \OO_{\Delta_Y}=\F\otimes_Y\G$ (where  for $F,G\in D(Y)$). Therefore $tor_i^{Y\times Y}(\F\boxtimes \G,\OO_{\Delta_Y})=tor_i^Y(F,G)$. In our case \[tor_i^{Y\times Y}(\OO_{\PP^k}\boxtimes \OO_{\PP^k}, \OO_{\Delta_Y})\cong tor_i^Y(\OO_{\PP^k},\OO_{\PP^k})=\wedge^i\N^\vee_{\PP^k/Y}.\] 
    
    Next we compute the difference between $tor_i^{X\times X\times Y\times Y}(\OO_{Z_{13}\times Z_{24}},\OO_{X\times X\times\Delta_{34}\, Y})$ \hfill\break and $tor_i^{X\times X\times Y\times Y}(\OO_{E_{13}\times E_{24}},\OO_{X\times X\times\Delta_{34}\, Y})$. This is achieved by tensoring with $\OO_{X\times X\times \Delta_{34} Y}$ the two exact sequences
    \[0\rightarrow \OO_{Z_{13}}\boxtimes \OO_{Z_{24}}(-E_{24})\rightarrow \OO_{Z_{13}\times Z_{24}}\rightarrow \OO_{Z_{13}\times E_{24}}\rightarrow 0\]
    \[0\rightarrow\OO_{Z_{13}}(-E_{13})\boxtimes \OO_{E_{24}}\rightarrow \OO_{Z_{13}\times E_{24}}\rightarrow \OO_{E_{13}\times E_{24}}\rightarrow 0\]
    The assertion follows after a little calculation with the first exact sequence.
    \end{proof}
    
 Concerning the underived tensor product, since
 \[tor_0(\OO_{Z_{13}}\boxtimes \OO_{Z_{24}},\OO_{X\times X\times \Delta_{34}Y})=\OO_{(\Delta_{12,34}\,Z)\cup (\PP^l\times\PP^l\times\Delta_{34}\PP^k)}\]
 we have the "Mayer-Vietoris" exact sequence
    \[0\rightarrow tor_0(\OO_{Z_{13}}\boxtimes \OO_{Z_{24}},\OO_{X\times X\times \Delta_{34}Y})\rightarrow \OO_{\Delta_{13,24}Z}\oplus (\OO_{\PP^l\times\PP^l\times\Delta_{34}\PP^k})\rightarrow \OO_{\Delta_{12}\PP^l\times\Delta_{34}\PP^k}\rightarrow 0\]
    Since
    \[\omega_Z|_E=\OO(-l,-k)\]
    from the Claim we get that, for $i>0$
    \begin{equation}\label{tori}tor_i^{X\times X\times Y\times Y}(\OO_{Z_{13}}\boxtimes\omega_{Z_{24}},\OO_{X\times X\times\Delta_{34}\, Y})=
    \OO_{\PP^l\times\PP^l\times\PP^k}(0,-l,-k+i)^{\oplus {{l-1}\choose i}}
    \end{equation}
    (in particular, it vanishes for $i>l-1$). For $i=0$ we have the exact sequence
     \begin{equation}\label{tor0}0\rightarrow tor_0(\OO_{Z_{13}}\boxtimes \omega_{Z_{24}},\OO_{X\times X\times \Delta_{34}Y})\rightarrow \omega_{\Delta_{13,24}Z}\oplus \OO_{\PP^l\times\PP^l\times\Delta_{34}\PP^k}(0,-l,-k)\rightarrow \OO_{\Delta_{12}\PP^l\times\Delta_{34}\PP^k}(-l,-k)\rightarrow 0
     \end{equation}
     Applying ${p_{X\times X}}_*$, i.e. ${p_{12}}_*$, to  (\ref{tor0}) it follows easily that in any case 
      \begin{equation}\label{0}R^j {p_{12}}_*(   tor_0^{X\times X\times Y\times Y}(\OO_{Z_{13}}\boxtimes\omega_{Z_{24}},\OO_{X\times X\times\Delta_{34}\, Y}))=R^j {p_{12}}_*(\omega_{\Delta_{13,24}Z})=\begin{cases} \omega_{\Delta _{12}X}&\hbox{for $j=0$ and $i=0$}\\
     0&\hbox{otherwise}\\
     \end{cases}
     \end{equation}
  because $p_{12}$ restricted to $\Delta_{13,24}Z$ is simply the birational morphism $p:Z\rightarrow X$. Hence the above $tor_0$ does not cause any obstruction to the validity of (\ref{b'}). On the other hand, applying ${p_{12}}_*$ to (\ref{tori}) one sees that the vanishing
     \begin{equation}\label{i}{R^jp_{12}}_*(   tor_i^{X\times X\times Y\times Y}(\OO_{Z_{13}}\boxtimes\omega_{Z_{24}},\OO_{X\times X\times\Delta_{34}\, Y}))=0\quad\hbox{for all $i>0$ and all $j$}
     \end{equation}
    holds if and only if $k\ge l$. 
     Via an easy spectral sequence, (\ref{0}) and (\ref{i}) prove that (\ref{b'}), i.e.  full-faithfulness of $\Phi_{\OO_Z}^{X\rightarrow Y}:D^b(X)\rightarrow D^b(Y)$, holds if $k\ge l$. In a similar way it follows also that the full-faithfulness does not hold for $k<l$.
    \end{example}
    
    \begin{example} (\emph{Mukai flop}) We follow the notation of \cite{huybrechts}, \S11.4. We have the diagram
      \begin{equation}\xymatrix{&&\>\>\>\>\>\>\>\>\>\>\>\>\>\>\>\>\>\>\>\>{E}\subset\mathbb{\PP}\times\mathbb{\PP}^\vee\ar[ddll]_{\pi_X}\ar@{^{(}->}[d]\ar[rrdd]^{\pi_Y}\\
  &&Z\ar[ld]_{p}\ar[rd]^{q}\\
  \mathbb{P}\>\>\>\ar@{^{(}->}[r]&X&&Y&\>\>\>\mathbb{P}^\vee\ar@{_{(}->}[l]\\}
  \end{equation}
  and $\N_{\PP|X}=\Omega_{\PP}$, $\N_{\PP^\vee|Y}=\Omega_{\PP^\vee}$. Here $\dim X=2n$ and $\PP=\PP^n$. The maps $p$ (resp. $q$) is the blow-up of $\PP$ (resp. $\PP^\vee$) and $E=\PP(\Omega_\PP)\subset \PP\times\PP^\vee$ is the incidence correspondence point-hyperplane. 
   It is well known, by a result of Kawamata and  Namikawa (\cite{kawa},\cite{nami}) that: \emph{ the functor ${q}_*\circ p^*=\Phi_{\OO_Z}:\DD^b(X)\rightarrow \DD^b(Y)$ is not fully faithful}. 
   
   \noindent Let us check this within the method of the previous example.  Exactly as above, the condition for full-faithfulness is (\ref{b'}), and one has to compute  
  \begin{equation}\label{torsmukai} tor_i^{X\times X\times Y\times Y}(\OO_{Z_{13}}\boxtimes \omega_{Z_{24}},\OO_{X\times X\times\Delta_{34}\, Y})
  \end{equation} 
  Again the intersection  in $X\times X\times Y\times Y$ of $Z_{13}\times Z_{24}$ and $X\times X\times \Delta_{Y_{34}}$ is
  the fibered product $Z\times_Y Z$, which has the two irreducible components: 
 \[(Z_{13}\times Z_{24})\cap (X\times X\times \Delta_{Y_{34}})=Z\times_Y Z=\Delta_{13,24}Z\cup (E_{13}\times_{\PP^\vee} E_{24}) \]
  One can compute all $tor$ sheaves (\ref{torsmukai}) as in Claim \ref{claim} and the result is similar. It happens that higher tor's (i.e. the sheaves (\ref{torsmukai}) for $i>0$) don't affect condition (\ref{b'}), namely
     \begin{equation}\label{imukai}{R^jp_{12}}_*(   tor_i^{X\times X\times Y\times Y}(\OO_{Z_{13}}\boxtimes\omega_{Z_{24}},\OO_{X\times X\times\Delta_{34}\, Y}))=0\quad\hbox{for all $i>0$ and all $j$}
     \end{equation}
   We leave this to the reader. 
  
  The reason why (\ref{b'}) is not satisfied is in the underived tensor product 
  \[tor_0^{X\times X\times Y\times Y}(\OO_{Z_{13}}\boxtimes \omega_{Z_{24}},\OO_{X\times X\times\Delta_{34}\, Y})\]
    As in the previous example this sits in the exact sequence

      \begin{equation}\label{mv}0\rightarrow tor_0(\OO_{Z_{13}}\boxtimes \omega_{Z_{24}},\OO_{X\times X\times \Delta_{34}Y})\rightarrow \omega_{\Delta_{13,24}Z}\oplus \bigl((p_{24}^*\omega_{Z})_{|E_{13}\times_{\Delta_{24}\PP^\vee} E_{34}})\bigr)\rightarrow (\omega_{\Delta_{13,24}Z})_{|\Delta_{13,24}E}\rightarrow 0
      \end{equation}
      We apply  ${p_{X\times X}}_*$, i.e. ${p_{12}}_*$ to the above exact sequence. Since $(\omega_X)_{|\PP}$ is trivial, we have that
      \begin{equation}\label{omega}(\omega_Z)_{|E}=\omega_E(-E)= \OO_{\PP\times\PP^\vee}(-n,-n)_{|E}\otimes\OO_{\PP\times \PP^\vee}(1,1)_{|E}= \OO_{\PP\times\PP^\vee}(-(n-1),-(n-1))_{|E}
      \end{equation}
      It follows that,  for all $i$, $R^i{p_{12}}_*$ applied to the sheaf on the right of the  exact sequence (\ref{mv}) is zero for all $i$. Therefore, to compute the higher direct images $R^i{p_{12}}_*$ of the $tor_0$ on the left, it is enough to compute $R^i{p_{12}}_*$ of the sheaf in the middle. This has two summands. Concerning the first 
      one, as in the previous example there is nothing contradicting (\ref{b'}), since
      \begin{equation}\label{ebbasta!}
      R^i{p_{12}}_*(\omega_{\Delta_{13,24}Z})=\begin{cases}\omega_{\Delta_{12}X}&\hbox{for $i=0$}\\0&\hbox{otherwise}\end{cases}
      \end{equation}
       Concerning the second summand, note that the fiber of the projection 
     \begin{equation}\label{projection}p_{12}: E_{13}\times_{\Delta_{34}\PP^\vee} E_{24}\rightarrow \PP\times\PP\subset X\times X
     \end{equation}
       over a pair $(x,x^\prime)\in\PP\times \PP$, with $x\ne x^\prime$, is the  intersection of the two hyperplanes of $\PP^\vee$ corresponding to $x$ and $x^\prime$, that is a $\PP^{n-2}\subset \PP^\vee$. Now (\ref{omega}) tells that $p_{24}^*\omega_Z$, restricted to a general fiber of  (\ref{projection})
     is $\OO_{\PP^{n-2}}(-(n-1))$.  Therefore $R^i{p_{12}}_*$ applied to the second summand of the middle part of sequence (\ref{mv}) is  zero for $i<n-2$ and \emph{non-zero and supported on $\PP\times \PP$} for $i=n-2$. By an easy spectral sequence this, together with (\ref{imukai}) and (\ref{ebbasta!}), yields that 
      $R^{n-2}\Phi_{\OO_Z\boxtimes_Y\omega_Z}(\OO_{\Delta_Y})$ is non-zero. Therefore $(\ref{b'})$ is not verified and $\Phi_{\OO_Z}:D^b(X)\rightarrow D^b(Y)$ is not fully faithful.
      
        \end{example}
      With a similar, but more complicated, calculation one can prove directly the  result of Kawamata and Namikawa (\cite{kawa},\cite{nami}, see also \cite{huybrechts}) that $\Phi_{\OO_{\widetilde Z}}:\DD^b(X)\rightarrow \DD^b(Y)$ is fully faithful, where $\widetilde Z=Z\cup (\PP\times \PP^\vee)$. 
      As a disclaimer, we should point out that this method, applied to the stratified Atiyah flop and Mukai flop (\cite{cautis},\cite{kawa2}, \cite{markman}, \cite{nami2}) becomes much more complicated.

 \section{Full-faithfulness via condition (c)}
 
 In this section we use condition (c) of Theorem \ref{gv} to provide another way to check full-faithfulness. So far condition (c) has proven to be extremely useful for detecting generic vanishing when the kernel  is a Poincar\'e line bundle. In fact, if $X$ is an abelian variety, $\cP$ a Poincar\'e line bundle on $X\times \widehat X$ (where $\widehat X$ is the dual abelian variety)  and $A$ is an ample line bundle on $\widehat X$,
 the object $\Phi^{\widehat X\rightarrow X}_{\cP^\vee}(A)$, has a peculiar description, which can be seen as an effect of the "abelianity" of the context: $\Phi^{\widehat X\rightarrow X}_{\cP^\vee}(A)$ is a locally free sheaf  which is, up to pullback via the isogeny $\varphi_A:\widehat X\rightarrow X$ associated to $A$, sum of copies of the line bundle $A^\vee$ (see e.g. \cite{mukai}, Prop. 3.11(1)). Therefore, in the case of Poincar\'e kernel on dual abelian varieties, condition ($c$) is a very effective way of reducing the GV condition to  vanishing theorems. This idea, due to Hacon (\cite{hacon}), is extremely frutful in the study of the geometry of irregular varieties. It is an interesting problem to find an adequate description of the objects  $\Phi^{Z\rightarrow Y}_{\cP^\vee}(A)$ in other cases. 
 
 
 In the present context, condition (c) of Theorem \ref{gv}   leads to the  full-faithfulness criterion below. Due to the above reason, at present its range of applicability is confined to abelian or irregular varieties. 
 \begin{proposition}\label{c} Let $A$ be a sufficiently high power of an ample line bundle on $X\times X$. Then $\Phi_\E:D^b(X)\rightarrow D^b(Y)$ is fully faithful if and only if 
  \begin{equation}\label{hyp}h^i(Y, \, \Phi^{X\times X\rightarrow Y}_{\E\boxtimes_Y\E^\vee}(A)\otimes\omega_Y)= \begin{cases}0&\hbox{for $i\ne \dim Y$}\\
 h^0(A\otimes \OO_{\Delta_X})&\hbox{for $i=\dim Y$}\\
 \end{cases}
 \end{equation}
 \end{proposition} 
 \begin{proof} 
 By the equivalence between (b) and (c) of Theorem \ref{gv} the first line above means that $\Phi_{\E\boxtimes_Y\E^\vee}(\omega_Y)$ is a sheaf in cohomological degree $\dim Y$, denoted  $\widehat\omega_Y[-\dim Y]$ (according to notation (\ref{hat})). Therefore the adjunction morphism (\ref{left}) is, up to a shift
 \[\Phi_{\widehat{\omega_Y}}^{X\rightarrow X}\rightarrow \Phi_{\Delta_X}^{X\rightarrow X}\]
 This induces a morphism of $\OO_{X\times X}$-modules 
 \begin{equation}\label{morphism}
 \widehat{\omega_Y}\rightarrow \OO_{\Delta_X}
 \end{equation}
  which is surjective since, for all $x\in X$, the adjunction morphism $\Phi_{\widehat{\omega_Y}}^{X\rightarrow X}(k_x)\rightarrow k_x$ is non-zero, hence surjective. 
  
  We stop for a moment, to recall from \cite{pp4}, Lemma 2.1 the functorial isomorphism, for all $i$ and for all objects $G$ (resp. $A$) of $D^b(Y)$ (resp. $D^b(Z)$)
  \begin{equation}\label{leray}H^i(Y,G\otimes \Phi^{X\times X\rightarrow Y}_{\E\boxtimes_Y\E^\vee}(A))\cong H^i(X\times X, \Phi_{\E\boxtimes_Y\E^\vee}^{Y\rightarrow X\times X}(G)\otimes A) \>\footnote{this is also the key ingredient in the proof of the equivalence between $(b)$ and $(c)$ of Theorem \ref{gv}}\end{equation}
Note that, via duality, (\ref{leray}) is a restatement of the description of the adjoints of Fourier-Mukai functors, but it is more simply proved by the fact that
\[R\Gamma(Y, G\otimes \Phi^{X\times X\rightarrow Y}_{\E\boxtimes_Y\E^\vee}(A))\cong R\Gamma(Y\times X\times X, p_Y^*G\otimes (\E\boxtimes_Y\E^\vee)\otimes p_{X\times X}^*(A)\cong R\Gamma(X\times X, \Phi_{\E\boxtimes_Y\E^\vee}^{Y\rightarrow X\times X}(G)\otimes A)\]
by Leray isomorphism and projection formula. 

Going back to our proposition, given a line bundle $A$ which is a sufficiently high power of an ample line bundle, from (\ref{leray}) and the second line of  (\ref{hyp}) we get
\[ h^0(Y, \, \Phi^{X\times X\rightarrow Y}_{\E\boxtimes_Y\E^\vee}(A)\otimes   \omega_Y)=h^0(X\times X,\widehat \omega_Y\otimes A)=h^0(X\times X,\OO_{\Delta_X}\otimes A)\]
Therefore,  since the morphism (\ref{morphism}) is surjective, Serre's vanishing applied to its kernel yields that
  (\ref{morphism}) is an isomorphism. Hence (\ref{left2}) is verified.  This proves that (\ref{hyp}) implies full-faithfulness of $\Phi_\E^{X\rightarrow Y}$. 
  The other implication  follows immediately from (\ref{leray}) and Serre's vanishing. \end{proof}

  \begin{example} (\emph{Mukai's theorem on the Poincar\'e kernel})  To illustrate Prop. \ref{c} let us take $X$ an abelian variety, $Y=\widehat X$ and as $\cP$ a Poincar\`e line bundle. We will show that $\Phi_\cP^{X\rightarrow \widehat X}$ is fully faithful. From this it follows that it is in fact an equivalence. Moreover, since $\cP^\vee=(-id,id)^*=(id,-id)^*\cP$, this proves also that \[
  \Phi_\cP^{X\rightarrow \widehat X}\circ \Phi_\cP^{\widehat X\rightarrow X}=(-id)^*[-\dim X]
  \]
  i.e. the theorem of Mukai in \cite{mukai}.    As in \cite{huybrechts} Prop. 9.19, in characteristic zero one has a much easier proof,  using Bondal-Orlov's strong simplicity criterion (see \S3) . The present proof works in any characteristic, as well as Mukai's original proof.
  
  Let $L=\OO_X(n\Theta)$  be a sufficiently high power ample line bundle on $X$. We take $A=L\boxtimes L$. By Proposition \ref{c} it is sufficient to prove that 
  \begin{equation}\label{goal}h^i(\widehat X, \Phi^{X\times X\rightarrow \widehat X}_{\cP\boxtimes_{\widehat X}\cP^\vee}(L\boxtimes L))=\begin{cases}0&\hbox{if $i\ne\dim X$}\\ h^0(X,L^2)&\hbox{if $i=\dim X$}\\\end{cases}
  \end{equation}
  Let 
  \[\varphi_L:X\rightarrow \widehat X\]
   the isogeny associated to $L$. We have that
   \[(id,\varphi_L)^*\cP=(p_1+p_2)^*L\otimes p_1^*L^{-1}\otimes p_2^*L^{-1}\]
   where $p_1+p_2:X\times X\rightarrow X$ is the group law. Therefore, letting $\F=(p_1+p_2)^*L\otimes p_1^*L^{-1}\otimes p_2^*L^{-1}$, by flat base change we have that  
   \begin{equation}\label{etale}\varphi_L^*\bigl(\Phi_{\cP\boxtimes_{\widehat X}\cP^\vee}^{X\times X\rightarrow \widehat X}(L\boxtimes L)\bigr)=\Phi_{\F\boxtimes_X\F^\vee}^{X\times X\rightarrow X}(L\boxtimes L)
   \end{equation}
   where $U\boxtimes_X V$  means $p_{12}^*U\otimes p_{23}^*V$, where $p_{12}$ and $p_{23}$ are the two projections of $(X\times X)\times_X(X\times X)=X\times X\times X$ (the fibred product is with respect to the second projection of the first factor and the first projection of the second factor).
   Therefore we must compute the right-hand side of (\ref{etale}). We have that
   \[\F\boxtimes_X\F^\vee=(p_1+p_2)^*L\otimes (p_2+p_3)^*L^{-1}\otimes p_1^*L\otimes p_3^*L^{-1}\]
   Therefore 
   \begin{equation}\label{conto1}\Phi_{\F\boxtimes_{\widehat X}\F^\vee}^{X\times X\rightarrow  X}(L\boxtimes L)={p_2}_*\bigl((p_1+p_2)^*L\otimes (p_2+p_3)^*L^{-1}\otimes p_1^*L^2\bigr)
   \end{equation}
   and by Serre vanishing (actually this is not needed in the case) the outcome is a sheaf:
   \[{p_2}_*\bigl((p_1+p_2)^*L\otimes (p_2+p_3)^*L^{-1}\otimes p_1^*L^2\bigr)=R^0{p_2}_*\bigl((p_1+p_2)^*L\otimes (p_2+p_3)^*L^{-1}\otimes p_1^*L^2\bigr)[0]\]
 Via the automorphism $(p_1,p_1+p_2, p_2+p_3):X\times X\times X\rightarrow X\times X\times X$, i.e. $(x,y,z)\mapsto (x,x+y,y+z)$, the 
   line bundle $(p_1+p_2)^*L\otimes (p_2+p_3)^*L^{-1}\otimes p_1^*L^2$ is identified to $L^2\boxtimes L\boxtimes L^{-1}$. Hence
   \[h^i(X,\Phi_{\F\boxtimes_{\widehat X}\F^\vee}^{X\times X\rightarrow  X}(L\boxtimes L))=h^i(X\times X\times X, L^2\boxtimes L\boxtimes L^{-1})= \begin{cases}h^0(X,L)^2\,h^0(X,L^2)&\hbox{for $i=\dim X$}\\
      0&\hbox{otherwise}
      \end{cases}\]
      Since the degree of the isogeny $\varphi_L$ is $ h^0(X,L)^2$, we get (\ref{goal}).
  \end{example}

\providecommand{\bysame}{\leavevmode\hbox
to3em{\hrulefill}\thinspace}

\end{document}